\documentclass{amsart}
\usepackage{amsmath,amsthm,amssymb}
\usepackage{xcolor}
\usepackage{breqn}
\usepackage{verbatim}
\usepackage{hyperref}

\usepackage{graphicx}
\usepackage{subcaption}

\usepackage{mathtools}
\mathtoolsset{showonlyrefs}

\newtheorem{theorem}{Theorem}[section]

\newtheorem{proposition}[theorem]{Proposition}
\newtheorem{definition}[theorem]{Definition}

\newtheorem{question}[theorem]{Question}

\def \R{\mathbb R}
\def \B{\mathbb B}
\newcommand \vol[2][n]{{\left|#2\right|}_{#1}}

\def \bKt{\overline{K_t}}
\def \bK{\bar K}
\title{On the volume of convolution bodies in the plane}
\author{J. Haddad}
\begin{document}
\begin{abstract}
	For every convex body $K \subset \R^n$ and $\delta \in (0,1)$, the $\delta$-convolution body of $K$ is the set of $x \in \R^n$ for which $\vol{K \cap (K+x)}\geq \delta \vol{K}$.
	We show that for $n=2$ and any $\delta \in (0,1)$, ellipsoids do not maximize the volume of the $\delta$-convolution body of $K$, when $K$ runs over all convex bodies of a fixed volume.
	This behavior is somehow unexpected and contradicts the limit case $\delta \to 1^-$, which is governed by the Petty projection inequality.
\end{abstract}
\maketitle

\section{Introduction}
Let $K \subseteq \R^n$ be a convex body (compact, convex and with non-empty interior) and let $g_K(x) = \vol{K \cap (K+x)}$ denote the covariogram function, where $\vol{\ \cdot\ }$ is the $n$-dimensional Lebesgue measure.
For $\delta \in (0,1)$, the convolution body of $K$ of parameter $\delta$ is the set defined by
\[C_\delta K = \{x \in \R^n: g_K(x) \geq \delta \vol{K}\}.\]

The set $C_\delta K$ is called the convolution body of $K$, due to the fact that $g_K$ is the convolution of the indicator functions of $K$ and $-K$.
Convolution bodies and the covariogram function were studied in \cite{kiener1986extremalitat,MA, meyer1993volume, MS92, tsolomitis1997}. Specifically, in relation to the {\it phase retrieval problem} in Fourier analysis, it was studied in \cite{averkov2007retrieving, averkov2009confirmation, bianchi2005matheron}.

When $\delta \to 1^-$ the set $C_\delta K$ collapses to the origin. The shape of $C_\delta K$, if scaled by a factor $(1-\delta)^{-1}$, approaches the polar projection body of $K$ denoted by $\Pi^*K$, which is the unit ball of the norm defined by
\[\|v\|_{\Pi^*K} = \vol[n-1]{P_{v^\perp}K}\]
for every unit vector $v \in S^{n-1}$, where $P_{v^\perp}$ is the orthogonal projection to the hyperplane orthogonal to $v$.
This was first observed by Matheron in \cite{MA}, where the covariogram function was introduced.
Indeed it was proven in \cite[Theorem 2.2]{MS92} that
\begin{equation}
	\label{eq_schmuck_limit}
	\lim_{\delta \to 1^-} \frac {\vol{C_{\delta}K}}{(1-\delta)^n} = {\vol{\Pi^*K}}.
\end{equation}

The classical Petty projection inequality (see Section 10.9 of \cite{Sh1}) states that 
\begin{equation}
	\label{ineq-petty-projection}
	\vol{\Pi^*K}\leq \vol{\Pi^*B_K}
\end{equation}
where $B_K$ is the Euclidean ball with same volume as $K$.
Equality holds in \eqref{ineq-petty-projection} if and only if $K$ is an ellipsoid (an affine image of the Euclidean ball).
The left-hand side of inequality \eqref{ineq-petty-projection} is invariant under volume-preserving affine transformations.
This was proven by Petty in \cite{petty61_2}, and Schmuckenschläger gave a simpler proof of this fact using \eqref{eq_schmuck_limit} and the obvious fact that $C_\delta(\varphi(K)) = \varphi(C_\delta K)$ for every volume-preserving affine transformation $\varphi$.

At the opposite endpoint, $\delta \to 0^+$, the body $C_{\delta} K$ converges to the difference body of $K$, defined by
\[DK = \{x-y: x, y \in K\}.\]

By the Brunn-Minkowsky inequality (see \cite[Theorem 7.1.1]{Sh1}), $\vol{DK} \geq 2^n \vol{K}$, with equality if and only if $K$ is symmetric with respect to some point (i.e. $x_0 + K = x_0 - K$ for some $x_0 \in K$).
Since $B_K$ is origin-symmetric,
\begin{equation}
\label{ineq_C0}
\vol{DK} \geq \vol{D B_K},
\end{equation}
which is reverse to the inequality \eqref{ineq-petty-projection}.
Nevertheless, \eqref{ineq_C0} is an equality for all symmetric sets.

An extension of the Petty projection inequality to certain averages of volumes of $C_\delta K$ can be deduced from the results in \cite{kiener1986extremalitat}.
\begin{theorem}
	\label{res_weighted_inequality}
	For every non-decreasing function $\omega:[0,1] \to [0, \infty)$ and every convex body $K$,
	\[\int_0^1 \omega(\delta) \vol{C_\delta K} d\delta \leq \int_0^1 \omega(\delta) \vol{C_\delta B_K} d\delta.\]
\end{theorem}
The results in \cite{kiener1986extremalitat} follow from the well-known Riesz convolution inequality, and Theorem \ref{res_weighted_inequality} recovers the Petty projection inequality (without the equality case) thanks to \eqref{eq_schmuck_limit} and a limit argument.
Namely, one chooses $\omega$ to be an approximation of the Dirac delta at $1$.
Since $\omega$ must be non-decreasing, this argument cannot be applied to a Dirac delta at some other point in $(0,1)$.
A particular case of Theorem \ref{res_weighted_inequality} is that
\[\int_t^1\vol{C_\delta K} d\delta \leq \int_t^1 \vol{C_\delta B_K} d\delta\]
for any $t \in (0,1)$.

A second application of the Riesz convolution inequality to convex bodies defined from  $C_\delta K$, was given in \cite{haddad2022affine}.

A {\it radial set} is a set of the form
\[K = \{0\} \cup \{x \in \R^n \setminus \{0\}: |x| \leq \rho_K(x/|x|)\} \]
where $\rho_K :S^{n-1} \to [0,\infty)$ is continuous, and $|\,\cdot\,|$ is the Euclidean norm.
A {\it radial body} is a radial set for which $\rho_K$ is strictly positive.
Every convex body containing the origin is also a radial body.

For every convex body $K$ and $p > -1, p\neq 0$, the $p$-radial mean body of $K$ is the radial body defined by
\[\rho_{R_pK}(v) = \left( \int_0^1 \rho_{C_\delta K}(v)^p d \delta\right)^{1/p},\]
while $R_0 K$ is defined as a limit of the sets $R_p K$ when $p \to 0$.
The original definition given in \cite{GZ98} is different, but equivalent to ours.
This can be deduced easily from formulas (3), (16) and (17) in \cite{haddad2023affine}.
\begin{theorem}[{ \cite[Theorem 20]{haddad2022affine}}]
	\label{res_ludwig_haddad}
	For every convex body $K$ and $p \in (-1, n)$,
	\[\vol{R_p K} \leq \vol{R_p B_K}.\]
	For $p>n$ the inequality is reversed.
	Equality holds if and only if $K$ is an ellipsoid.
\end{theorem}
It was proven in \cite{GZ98} that $R_p K$ approaches $\Pi^*K$ when $p \to -1^+$, so Theorem \ref{res_ludwig_haddad} is yet an other extension of the Petty projection inequality involving averages of $C_\delta K$.

Theorems \ref{res_weighted_inequality} and \ref{res_ludwig_haddad} suggest the possibility that for a fixed $\delta \in (0,1)$, $\vol{C_\delta K}$ is also maximized by ellipsoids, among sets of a fixed volume.
Of course, due to \eqref{ineq_C0} this is only possible if we restrict the problem to the symmetric case, or to some range of $\delta \in (0,1)$ far from $0$.
Let us formulate the weakest possible question:

\begin{question}
	\label{thequestion}
	Is there a value of $\delta \in (0,1)$ such that
	\begin{equation}
		\label{ineq_thequestion}
		\vol{C_{\delta } K} \leq \vol{C_{\delta} B_K}
	\end{equation}
for every symmetric convex body $K$?
\end{question}
The purpose of this paper is to give a complete answer to this question in dimension $2$.

Observe that due to Theorem \ref{res_weighted_inequality}, inequality \eqref{ineq_thequestion} holds ``in average'' in $\delta$ for every $K$.

The following proposition describes the situation in which $K$ is far from the set of ellipsoids.
Define the Banach-Mazur distance between two convex bodies $K,L \subseteq \R^n$ as
\[d_{\operatorname{BM}}(K,L) = \min\{ \lambda > 0 : K-x \subseteq \Phi(L-y) \subseteq e^\lambda (K-x) \text{ for } \Phi \in GL(n), x,y \in \R^n\}.\]
where $GL(n)$ is the set of invertible linear transformations of $\R^n$.
Let $\B$ be the unit Euclidean ball in $\R^n$.
It follows from the definition that $d_{\operatorname{BM}}(K,\B) = 0$ if and only if $K$ is an ellipsoid.
\begin{proposition}
	\label{res_far_from_ball}
	For every convex body $K \subseteq \R^n$ which is not an ellipsoid, $\vol{C_{\delta} K} \leq \vol{C_{\delta} B_K}$ for every $\delta > \varphi( d_{\operatorname{BM}}(K,\B))$, where $\varphi:[0,\infty) \to (0,1]$ is a continuous function with $\varphi(t)=1$ if and only if $t=0$.
\end{proposition}
We will prove this fact in Section \ref{sec_first_order}.
Proposition \ref{res_far_from_ball} reduces the problem to a local question:
If \eqref{ineq_thequestion} is valid for every $K$ sufficiently close to the Euclidean ball and $\delta$ close to $1$, then thanks to Proposition \ref{res_far_from_ball}, it is valid for every $K$ and $\delta$ close to $1$.

\begin{definition}
	\label{def_Kt}
For any radial set $K$ we will consider a one-parameter family of radial bodies $K_t$ defined by 
	\begin{equation}
		\label{eq_def_Kt}
		\rho_{K_t}(v) = 1 + t \rho_K(v).
	\end{equation}
We also define
	\begin{equation}
		\label{eq_def_Ktbar}
		\bKt = K_t/\vol{K_t}^{1/n}.
	\end{equation}
\end{definition}
We will say that a radial set $K$ is $C^\beta$ smooth with $\beta\geq 1$ if the radial function $\rho_K$ is $C^\beta$. Notice that this definition does not coincide with the smoothness of the set $\partial K$ as usual, because we are allowing $\rho_K(v)=0$.
But it is clear that if $K$ is $C^\beta$ smooth, then $K_t$ has a $C^\beta$ smooth boundary in the usual sense.

We will analyze $\vol{C_{\delta} \bKt}$ as a function of $t$ and $\delta$, for $t$ near $0$. First we obtain:
\begin{theorem}
	\label{res_first_variation}
	For every $C^1$ radial set $K \subseteq \R^n$ and $\delta \in (0,1)$, the function $t \mapsto \vol{C_{\delta} \bKt}$ is $C^1$ and we have 
	\[  \frac {\partial}{\partial t} \vol{C_{\delta} \bKt}\bigg|_{t=0} = 0.\]
\end{theorem}

Then it suffices to analyze the second derivative of $t \mapsto \vol{C_{\delta} \bKt}$.
In the limit $\delta \to 1^-$ this second derivative is completely described in Section \ref{sec_limit} for $n=2$, and its sign is compatible with the fact that $t \mapsto \vol{\Pi^*\bKt}$ has a maximum at $t=0$.
In Section \ref{sec_limit} we show:
\begin{theorem}
	\label{res_positive_almost}
	For every $C^2$ smooth radial set $K \subseteq \R^2$ the function $t \mapsto \vol{C_{\delta} \bKt}$ is $C^2$ for every $\delta \in (0,1)$ and
	\begin{align}
		\lim_{\delta \to 1^-} \frac 1{(1-\delta)^2}  \frac{\partial^2}{\partial t^2} \vol[2]{C_\delta \bKt} \bigg|_{t=0}
		&\leq 0.
	\end{align}
	Equality holds if and only if $\rho_K$ is the restriction of a polynomial of degree $2$ to the unit circle.

\end{theorem}
The equality cases of Theorem \ref{res_positive_almost} correspond to variations $\bKt$ that coincide up to first order with families of ellipsoids.

At this point it is natural to expect that Theorem \ref{res_positive_almost} combined with an approximation argument and Proposition \ref{res_far_from_ball}, could yield a positive answer to Question \ref{thequestion}.
However, for this argument to be complete we need the convergence of the second derivatives of the volume as $\delta \to 1^-$, to be uniform with respect to $K$.
We were unable to show this uniform convergence, and the following counterexample shows why:

\begin{theorem}
	\label{res_cos_m_perturbation}
	Let $K^m \subseteq \R^2$ be the (symmetric) radial set defined by $\rho_{K^m}(v) = \cos(2 m v)^2$ with $v \in [0,2\pi]$.
	Then for every $\delta \in (0,1)$ there exists $m \in \mathbb N$ such that 
	\[  \frac {\partial^2 }{\partial t^2}\vol[2]{C_{\delta} \overline{K_t^m}} \bigg|_{t=0} > 0.\]
\end{theorem}

As a consequence, we get a negative answer to Question \ref{thequestion} in dimension $2$, and every value of $\delta \in (0,1)$.
\begin{theorem}
	\label{res_negative_answer}
	For every $\delta \in (0,1)$ there exists a symmetric convex body $K \subseteq \R^2$ such that $\vol{C_{\delta} K} > \vol{C_{\delta} B_K}$.
	Moreover, $K$ can be chosen arbitrarily close to the Euclidean ball in the $C^\infty$ topology.
\end{theorem}

It is important to remark that for a fixed $m$ in Theorem \ref{res_cos_m_perturbation}, the set of $\delta \in (0,1)$ for which $  \frac {\partial^2 }{\partial t^2}\vol[2]{C_{\delta} \bKt} \bigg|_{t=0}$ is positive, is a complicated union of intervals that grow in number and accumulate near $1$, as $m \to \infty$ (see Section \ref{sec_second_order}). 
Previous attempts to find regular polygons that are counterexamples to Question \ref{thequestion} for $\delta$ close to $1$ by direct computation, failed probably because of this complicated behaviour.
We still do not know if regular polygons are counterexamples to Question \ref{thequestion}.

The following natural question remains open:
\begin{question}
	For each fixed $\delta \in (0,1)$, what convex bodies are maximizers of $C_\delta K$ when $K$ runs among sets of the same volume?
\end{question}

The rest of the paper is organized as follows:

In Section \ref{sec_notation} we introduce all the notation that will be necessary for our computations in the following sections.
In Section \ref{sec_preliminary} we obtain the results concerning convex sets far from the ball (Proposition \ref{res_far_from_ball}), and establish several technical lemmas that will be needed later.

In Section \ref{sec_first_order} we compute the first-order approximation of $\vol{C_\delta \bKt}$ at $t=0$ (Theorem \ref{res_first_variation}). All results in this section are proved in dimension $n$.

In Section \ref{sec_second_order} we compute the second order approximation in the plane, and establish Theorems \ref{res_cos_m_perturbation} and \ref{res_negative_answer}.

Finally in Section \ref{sec_limit} we compute the limit of the second derivative of $\vol{C_\delta \bKt}$ when $\delta \to 1^-$, and prove Theorem \ref{res_positive_almost}.

\section{Notation}
\label{sec_notation}
The closed Euclidean ball of center $p \in \R^n$ and radius $r>0$ will be denoted by $B(p, r)$.
The closed unit Euclidean ball $B(0,1)$ is denoted by $\B$, and its volume, by $\omega_n$.

It is convenient to introduce some notation in order to simplify the lengthy computations that we will carry over in Sections \ref{sec_first_order} and \ref{sec_second_order}.
The following notation is by no means standard.

For any set $L$ and $x \in \R^n$, we denote $G_L(x) = L \cap (L+x)$.
For $x \in \R^n$ denote $L(x) = \B \cap (\B+x) = G_{\B}(x)$, $C(x) = S^{n-1} \cap (\B+x)$ and $S(x) = C(x) \cup C(-x)$.

The $n-1$ dimensional volume $S(s) = \vol[n-1]{S(s v)}$ is independent of $v \in S^{n-1}$ for any $s>0$, as well as the $n$-dimensional volume $L(s) = \vol{L(s v)}$.

For a fixed radial set $K$, $v \in S^{n-1}$ and $\delta \in (0,1)$, denote 
\[k(t) = \vol{K_t},\quad \rho_v(t) = \rho_{C_\delta \bKt}(v),\quad s_v(t) = \rho_v(t) k(t)^{1/n} \text{ and } g_v(t,s) = g_{K_t}(s v).\]
These quantities depend on the set $K$ which is not explicitly written in the notation.

The partial derivatives of a function $g(t,s)$ will be denoted by $\partial_s g_v, \partial_t g_v, \partial_{s,t} g_v$ and so on.
We will denote $k_0 = k(0), k_0' = k'(0), k_0'' = k''(0)$.

For $A \subseteq S^{n-1}$ and functions $f,g :S^{n-1} \to \R$ it will be convenient to use:
\[ [f,g]_A = \{t y \in \R^n: y \in A, t \in [f(y), g(y)]\},\]
and for $x\in\R^n$,
\[ [f,g]_A^x = [f,g]_A + x.\]

With this notation we have $K = [0,\rho_K]_{S^{n-1}}$ and $G_K(x) = [0,\rho_K]_{S^{n-1}} \cap [0,\rho_K]_{S^{n-1}}^x$.

The union of two disjoint sets will be denoted by $A \sqcup B$ to emphasize that $A \cap B = \emptyset$.

To measure the parameter of the convolution bodies $C_\delta K$ we will use the three different variables $\delta \in (0,1), s \in (0,2)$ and $\alpha \in (0, \pi/2)$, related by the formulas
\begin{equation}
	\label{eq_relations_alpha_delta_s}
	\delta = L(s), s = 2 \cos(\alpha).
\end{equation}

Our computations will involve the quantities 
\[W_{K,v}(s) = \int_{S(s v)} \rho_K(w) dw,\ \  I_K = \int_{S^1}\rho_K(w)dw.\]
In the variable $\alpha$ we will denote $w_{K,v}(\alpha) = W_{K,v}(2 \cos(\alpha))$.

For the computations in $\R^2$ we will identify points in $S^1$ with their angle in $[0,2\pi)$, and write indistinctly $\rho_K(v)$ for $v \in [0,2\pi)$ or $v \in S^1 \subseteq \R^2$.
We will also use the vector $v_\alpha = (\cos(\alpha), \sin(\alpha))$.

\section{Preliminary results}
\label{sec_preliminary}

We start by proving Proposition \ref{res_far_from_ball} and Theorem \ref{res_weighted_inequality}.
\begin{proof}[Proof of Proposition \ref{res_far_from_ball}]
    According to \cite[Corollary 2]{MS92}, for every convex body $K \subseteq \R^n$
	\[(1-\delta)^n \vol{\Pi^*K}  \leq \vol{C_{\delta}K} \leq (-\log(\delta))^n \vol{\Pi^*K}. \]
	so
    \begin{align}
        \vol{C_{\delta}K}
	    &\leq (-\log(\delta))^n \vol{\Pi^*K} \\
	    &\leq (-\log(\delta))^n \frac{\vol{\Pi^*K}}{\vol{\Pi^*B_K}} \vol{\Pi^*B_K} \\
	    &\leq \left( \frac{-\log(\delta)}{1-\delta} \right)^n \frac{\vol{\Pi^*K}}{\vol{\Pi^*B_K}} \vol{C_{\delta}B_K}. \\
    \end{align}
	Assuming $K$ is not an ellipsoid, $\frac{\vol{\Pi^*K}}{\vol{\Pi^*B_K}}<1$ and we may find an appropriate $\delta_0(K)$ for which $\vol{C_{\delta}K} \leq \vol{C_{\delta}B_K}$ if $\delta > \delta_0(K)$.
	Indeed, by a theorem of Böröczky \cite[Corollary 5]{boroczky2013stronger}, there exists a constant $\gamma_n > 0$ such that
	\[\vol{\Pi^* K} \leq (1 - \gamma_n d_{\operatorname{BM}}(K, \B)^{1680 n}) \vol{\Pi^* B_K}\]
and we get
	\[\vol{C_{\delta}K} \leq \left(\frac{-\log(\delta)}{1-\delta}\right)^n  (1 - \gamma_n d_{\operatorname{BM}}(K, \B)^{1680 n})\vol{C_{\delta}B_K}.\]
	Using that $\frac{-\log(\delta)}{1-\delta} \leq \delta^{-1}$ for $\delta \in (0,1)$, 
	it suffices to take \\ $\delta_0(K) = (1 - \gamma_n d_{\operatorname{BM}}(K, \B)^{1680 n})^{1/n}$ and the function $\varphi(t) = (1 - \gamma_n t^{1680 n})^{1/n}$.
\end{proof}

\begin{proof}[Proof of Theorem \ref{res_weighted_inequality}]
In \cite[Section 2]{kiener1986extremalitat}, Kiener proves that for $p \geq 1$ and any convex body $K$,
\[\int g_K(x)^p dx \leq \int g_{B_K}(x)^p dx.\]

	A quick inspection of the proof (stated also in \cite[Lemma 3]{kiener1986extremalitat} for the equality case) shows that the $p$-th power can be replaced by any convex, non-negative and non-decreasing function $\varphi:[0,1] \to \R^+$, this is,
	\begin{equation}
		\label{ineq_kiener_convex}
		\int \varphi(g_K(x)) dx \leq \int \varphi(g_{B_K}(x)) dx.
	\end{equation}
	Assume without loss of generality that $\omega$ is $C^1$.
	Take $\varphi(t) = \int_0^1 \omega'(\delta) (t - \delta)_+ d \delta$ which is clearly non-negative, convex and non-decreasing.
	Using Fubini, integration by parts and the layer-cake formula,
	\begin{align}
		\int \varphi(g_K(x)) d x
		&= \int_0^1 \omega'(\delta) \int_{\R^n} (g_K(x) - \delta)_+ d x d\delta \\ 
		&= \int_0^1 \omega'(\delta) \int_\delta^1 \vol{C_s K} d s d\delta \\ 
		&= \omega'(0) \int_{\R^n} g_K(x) d x + \int_0^1 \omega(\delta) \vol{C_\delta K} d\delta. \\ 
		&= \omega'(0) \vol{K}^2 + \int_0^1 \omega(\delta) \vol{C_\delta K} d\delta. \\ 
	\end{align}
	By \eqref{ineq_kiener_convex} we get the result.
\end{proof}

The following technical proposition is essential to estimate $g_K(x)$ for small $x$.
Set
\[L(K, x) = L(x) \sqcup [1,\rho_{K}]_{C(x)} \sqcup [1,\rho_{K}]_{C(-x)}^x.\]
The sets $G_{K_t}(x)$ and $L(K_t,x)$ are very similar when $t>0$ is small, in fact they coincide outside a small region of volume $O(t^2)$, while the volume of $L(K_t, x)$ is easier to compute.

\begin{proposition}
	\label{res_serious_approximation}
	For $M>0$, $x\in\R^n, |x|<2$ there exist $c,t_0>0$ depending only on $M$ and $|x|$, such that for every radial set $K \subseteq \R^n$ with $\rho_K \leq M$ and for every $t \in (0, t_0)$,
	\[G_{K_t}(x) \setminus T(c t) = L(K_t, x) \setminus T(c t)\]
	where 
	\[T(t) = \{y \in \R^n: d(y, S^{n-1}\cap (S^{n-1}+x)) \leq t\}\]
	and $K_t$ is defined by \eqref{eq_def_Kt}.
\end{proposition}
\begin{proof}
	Let $E$ be the line parallel to $x$ passing through the origin, and $P$ the hyperplane perpendicular to $x$, passing through $x/2$.
	Denote by $a(y), b(y)$ the euclidean distances from $y$ to $P$ and $E$, respectively.
	Since $|x|<2$, we have $|x| = 2 \cos(\alpha)$ for a unique $\alpha \in (0,\pi/2)$.
	Consider the set
	\[U = \{y \in \R^n/ a(y) \leq \sin(\alpha)\}.\]
	It is clear that $U \subseteq L(x) \sqcup [0,\infty]_{C(x)} \sqcup [0,\infty]_{C(-x)}^x$.

	Now we claim that if $t \in \left(0,\frac{\sqrt{1+ 8 \cos(\alpha)^2}-1}M \right)$, then 
	\begin{equation}
		\label{eq_serious_claim}
		U \cap B(0, 1+M t) \cap B(x, 1+M t) \subseteq B(0, 1) \cup B(x, 1)
	\end{equation}
	(see Figure \ref{fig_serious})

	\begin{figure}
		\centering
		\caption{Set in the left-hand side of equation \eqref{eq_serious_claim}}
		\label{fig_serious}
		\includegraphics[width=.5\textwidth]{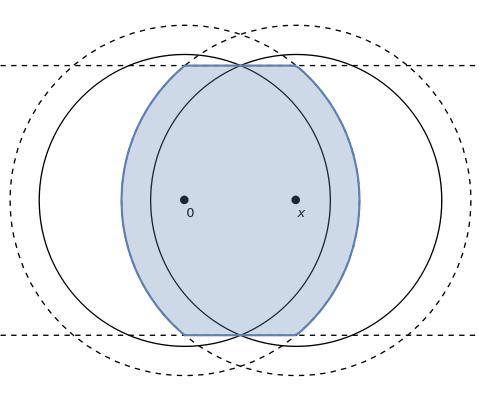}
	\end{figure}
	Indeed, the equations defining the left intersection are
	\begin{equation}
		\label{eq_serious_1}
		(a(y) + \cos(\alpha))^2 + b(y)^2 \leq (1+ M t)^2 ,
	\end{equation}
	\begin{equation}
		\label{eq_serious_2}
		b(y) \leq \sin(\alpha).
	\end{equation}
	If $a(y) \leq 2 \cos(\alpha)$, then $(a(y) - \cos(\alpha))^2\leq \cos(\alpha)^2$ and we get from \eqref{eq_serious_2}, 
	\[(a(y) - \cos(\alpha))^2+b(y)^2 \leq 1.\]
	If $a(y) \geq 2 \cos(\alpha)$, we also get
	\begin{align}
		(a(y) - \cos(\alpha))^2+b(y)^2
		&=(a(y) + \cos(\alpha))^2 + b(y)^2 -4 a(y) \cos(\alpha) \\
		&\leq (1+Mt)^2 -4 a(y) \cos(\alpha)\\
		&\leq 1+ 8 \cos(\alpha)^2-4 a(y) \cos(\alpha) \\
		&< 1,
	\end{align}
	and this implies in both cases that 
	\[y \in B(0, 1) \cup B(x, 1) \]
	and the claim is proven.

	Notice that
	\[B(0, 1) \cup B(x, 1) = L(x) \sqcup ([0,1]_{S^{n-1}}^x \cap [1,\infty]_{S^{n-1}}) \cap ([0,1]_{S^{n-1}} \cap [1,\infty]_{S^{n-1}}^x).\]

	Now, since both $G_{K_t}(x), L(K_t, x)$ lie inside $B(0, 1+M t) \cap B(x, 1+M t)$,
	a point in either of the sets $U \cap G_{K_t}(x), U \cap L(K_t, x)$ must belong to
	\begin{align}
		(B(0, 1) & \cup B(x, 1)) \cap (L(x) \sqcup [0,\infty]_{C(x)} \sqcup [0,\infty]_{C(-x)}^x)\\
			&= L(x) \sqcup ([0,1]_{S^{n-1}}^x \cap [1,\infty]_{C(x)}) \sqcup ([0,1]_{S^{n-1}} \cap [1,\infty]_{C(-x)}^x) 
	\end{align}

	Inside this set, it is clear that the conditions defining $G_{K_t}(x)$ and $L(K_t, x)$ coincide.
	To see this write $G_{K_t}(x) = [0,\rho_{K_t}]_{S^{n-1}} \cap [0,\rho_{K_t}]_{S^{n-1}}^x$ and
	\begin{align}
		G_{K_t}(x) & \cap \left( L(x) \sqcup ([0,1]_{S^{n-1}}^x \cap [1,\infty]_{C(x)}) \sqcup ([0,1]_{S^{n-1}} \cap [1,\infty]_{C(-x)}^x) \right) \\
		&= L(x) \sqcup ([0,1]_{S^{n-1}}^x \cap [1,\rho_{K_t}]_{C(x)}) \sqcup ([0,1]_{S^{n-1}} \cap [1,\rho_{K_t}]_{C(-x)}^x)  \\
		L(K_t, x) & \cap \left( L(x) \sqcup ([0,1]_{S^{n-1}}^x \cap [1,\infty]_{C(x)}) \sqcup ([0,1]_{S^{n-1}} \cap [1,\infty]_{C(-x)}^x) \right).
	\end{align}

	Then, we only need to prove that
		\[B(0, 1+M t) \cap B(x, 1+M t) \setminus U \subseteq T(c t).\]

	The equations defining the left-hand side, are \eqref{eq_serious_1} and
	\begin{equation}
		\label{eq_serious_not_2}
		b(y) \geq \sin(\alpha).
	\end{equation}
	From \eqref{eq_serious_1} we obtain
	\begin{equation}
		\label{eq_serious_3}
		a(y)^2 + \cos(\alpha)^2 + b(y)^2 \leq (1+ M t)^2,
	\end{equation}
	 and using \eqref{eq_serious_not_2} for $t<1$,
	 \begin{equation}
		 \label{eq_serious_4}
		 a(y)^2 \leq (1+M t)^2 - 1 = 2 M t + M^2 t^2 \leq (2M+M^2)t.
	 \end{equation}
	 
	 From \eqref{eq_serious_3} we also have 
	\[b(y)^2 \leq (1+Mt)^2-\cos(\alpha)^2 \leq (2 M + M^2) t +\sin(\alpha)^2,\]
	which yields, together with \eqref{eq_serious_4} and \eqref{eq_serious_not_2},
	 \begin{align}
		 (b(y)-\sin(\alpha))^2 + a(y)^2 
		 &= b(y)^2 -2 b(y) \sin(\alpha) + \sin(\alpha)^2 + a(y)^2 \\
		 &\leq (2M+M^2) t +2 \sin(\alpha)^2 -2 b(y) \sin(\alpha) + (2M+M^2)t\\
		 &\leq 2 (2M+M^2) t
	 \end{align}
	 for $t \in (0,1)$.

	 On the other hand, the equations defining $S^{n-1} \cap (S^{n-1}+x)$ are
	 \[a(y)=0, b(y)=\sin(\alpha),\]
	 and the equations defining $T(c t)$ are
	 \[(b(y)-\sin(\alpha))^2 + a(y)^2 \leq c t.\]

	 Thus, we have proved that $B(0, 1+M t) \cap B(x, 1+M t) \setminus U \subseteq T(c t)$ for $t \in (0,t_0)$, with $c=2(2M+M^2)$ and some $t_0$ small, and the proof is complete.

\end{proof}

The following proposition guarantees that the computations of first and second derivatives in the next section are correctly justified.
\begin{proposition}
	\label{res_gk_smooth}
	Let $K$ be a $C^\beta$ radial set with $\beta \geq 1$.
	Then there is $\varepsilon>0$ such that the function 
	\begin{align}
		S^{n-1} \times (-\varepsilon, \varepsilon) \times (0,2) & \to  \R \\
		(v, t, s) & \mapsto  g_v(t, s) = g_{K_t}(s v)
	\end{align}
	is $C^\beta$ smooth.
	Moreover, $\frac {\partial}{\partial s}g_v(t,s) \neq 0$.
\end{proposition}
\begin{proof}
	Fix $v_0 \in S^{n-1}, s_0 \in (0,2)$.
	Since $K$ is $C^\beta$ and $\partial K_0 = S^{n-1}$ intersects transversally with $\partial K_0 + s_0 v_0$, there is $\varepsilon>0$ small such that for all $(v,t,s)$ in an $\varepsilon$ neighborhood of $(v_0,0, s_0)$, the boundaries $\partial K_t$ and $\partial K_t+s v$ intersect transversally to each other, and to any line parallel to $v_0$ passing through a point in an $\varepsilon$ neighborhood of $G_{K_t}$.

	Let $P_{v,t,s}$ be the orthogonal projection of $G_{K_t}(s v)$ onto the plane orthogonal to $v_0$, $\langle v_0 \rangle^\perp$.
	By transversality, reducing $\varepsilon$ further if necessary, the set $G_{K_t}(s v)$ can be described as the region between the graphs of two functions $f_-$ and $f_+$. This is,
	\[G_{K_t}(s v) = \{y + l v_0 : y \in P_{v,t,s}, f_-(v,t,y) + s \leq l \leq f_+(v,t,y)\}\]
	for two $C^\beta$ functions $f_\pm$ defined for $(v,t)$ in a neighborhood of $(v_0, 0)$, and $y$ in a fixed open set containing $P_{v,t,s}$ for all such $(v,t,s)$.
	The volume can be computed as
	\begin{align}
		\label{eq_smooth_covariogram}
		g_v(t,s) 
		&= \int_{P_{v,t,s}} (f_+(v,t,y) - f_-(v,t,y) - s) d y\\
		&= \int_{S^{n-1} \cap \langle v_0 \rangle^\perp} \int_0^{\rho(v,t,s)} r^{n-2} (f_+(v,t,r \xi ) - f_-(v,t,r \xi ) - s) d r d\xi,
	\end{align}
	where $\rho(v,t,s)$ is the ($C^\beta$-smooth) radial function of $P_{v,t,s}$.
	Then it is clear that $g_v(t,s)$ is $C^\beta$ smooth around $(v_0, 0, s_0)$.

	By \eqref{eq_smooth_covariogram}, the partial derivative with respect to $s$ is exactly $-\vol[n-1]{P_{v,t,s}}$, which is non-zero since $s \in (0,2)$ implies $G_{K_t}$ has non-empty interior.
\end{proof}

\section{First-order Taylor Expansion of $C_\delta \bKt$}
\label{sec_first_order}

In order to compute the derivative of $\vol{C_\delta \bKt}$ we need to compute that of the covariogram function.

\begin{proposition}
	\label{res_d1_g}
	Let $K$ be a radial set and $K_t$ be the radial body defined by \eqref{eq_def_Kt}, then for $x \in \R^n$ with $0 < |x| < 2$,
	\begin{equation}
		\label{eq_d1_g}
		g_{K_t}(x) = \vol{L(x)} + t \int_{S(x)} \rho_{K}(v) d v + O(t^2).
	\end{equation}
	For $x=0$,
	\begin{equation}
		\label{eq_d1_k}
		\vol{K_t} = \vol{\B} + t I_K + O(t^2).
	\end{equation}
	Here $\frac{O(t^2)}{t^2}$ is bounded by a constant independent of $t \in (0,1)$ (but possibly depending on $K$ and $x$).
\end{proposition}

\begin{proof}
	Thanks to Proposition \ref{res_serious_approximation}, the set $G_{K_t}(x)$ can be approximated as the disjoint union
\begin{equation}
	\label{eq_separation}
	G_{K_t}(x)
	\sim  [1, \rho_K]_{C(x)} 
	\sqcup [1, \rho_K]_{C(-x)}^x
	\sqcup  L(x)
\end{equation}
	where $A \sim B$ means that the symmetric difference $A \Delta B$ has volume $O(t^2)$.
	Indeed, the symmetric difference must lie inside the torus $T(c t)$, whose volume is bounded by $c_n (c t)^2$ where $c_n$ is some dimensional constant.

	We obtain
	\[
		g_{K_t}(x) =\vol[2]{L(x)} + \vol[2]{[1, \rho_K]_{C(x)}} + \vol[2]{[1, \rho_K]_{C(-x)}} + O(t^2).
	\]
	Integrating in polar coordinates,
\begin{align}
	\vol[2]{[1, \rho_K]_{C(x)}} &+ \vol[2]{[1, \rho_K]_{C(-x)}} 
	= \vol[2]{[1, \rho_K]_{S(x)}} \\
	&= \frac 1n \int_{S(x)} (\rho_{K_t}(v)^n - 1) d v \\
	&= \frac 1n \int_{S(x)} (n t \rho_{K}(v) + O(t^2)) d v \\
	\label{eq_d2_C}
	&= t \int_{S(x)} \rho_{K}(v) d v + O(t^2),
\end{align}
and the proposition follows.
\end{proof}

\begin{proof}[Proof of Theorem \ref{res_first_variation}]
	For $t=0$, $\overline{K_0}$ is the Euclidean ball of volume $1$.
	The body $C_\delta \overline{K_0}$ is also a ball, and its radius $\rho_0$ satisfies $L(\rho_0) = \delta$.

Start observing that for any $\lambda>0$,
\[g_{\lambda K}(\lambda x) = \lambda^n g_K(x),\]
implying that
	\[g_{\bK}(x) = \vol{K}^{-1} g_K(\vol{K}^{1/n} x).\]

	Since $\rho_{C_\delta \bKt}(v) v$ is in the boundary of $C_\delta \bKt$, by the continuity of volume, the radial function $\rho_{C_\delta \bK}(v)$ satisfies
\[\delta = g_{\bKt}(\rho_{C_\delta \bKt}(v) v) = \vol{K_t}^{-1} g_{K_t}(\vol{K_t}^{1/n} \rho_{C_\delta \bKt}(v) v).\]

	We get
	\begin{equation}
		\label{eq_rho}
		\delta = k(t)^{-1} g_v(t, \rho_v(t) k(t)^{1/n}).
	\end{equation}
	Clearly, for $t$ close to $0$, the function $k(t)$ is $C^1$ smooth and bounded away from $0$.
	By Proposition \ref{res_gk_smooth} and the Implicit Function Theorem, the function $\rho_v(t)$ must be $C^1$ with respect to $(t,v)$, in a neighborhood of $t=0$.

	We can take derivative of \eqref{eq_rho} with respect to $t$, to obtain
\begin{align}
	\label{eq_d1_rho1}
	0 &= -k(t)^{-2} k'(t) g_v(t, \rho_v(t) k(t)^{1/n}) \\
	&+ k(t)^{-1} \partial_s g_v(t,\rho_v(t) k(t)^{1/n})\left(\frac 1n k(t)^{\frac 1n -1} k'(t) \rho_v(t) +k(t)^{1/n} \rho_v'(t) \right) \\
	&+ k(t)^{-1} \partial_t g_v(t,\rho_v(t) k(t)^{1/n})
\end{align}
	Notice that $g_v(0,s) = L(s)$ for every $s>0$, so $\partial_s g_v(0,s) = L'(s)$.
From \eqref{eq_d1_rho1} we compute
\begin{equation}
    \label{eq_d1_rho2}
    \rho_v'(0) = \omega_n^{-1/n} L'(s_0)^{-1}  \left( \omega_n^{-1} k'(0) L(s_0) - \partial_t g_v(0, s_0) \right) -  \frac 1n \omega_n^{-1} k'(0) \rho_0
\end{equation}
	where $\omega_n = k(0)$ and $s_0 = s(0) = \rho_v(0) k(0)^{1/n}$ is independent of $v$.

The volume of $C_\delta \bKt$ can be computed as
	\begin{equation}
		\label{eq_radial_volume}
		\vol{C_\delta \bKt} = \frac 1n \int_{S^{n-1}} \rho_{C_\delta \bKt}(v)^n d v
	\end{equation}
	so taking derivative with respect to $t$ and using \eqref{eq_d1_rho2} and \eqref{eq_d1_k},
\begin{align}
     \frac{\partial}{\partial t} \vol{C_\delta \bKt} \bigg|_{t=0}
	&= \rho_0^{n-1} \int_{S^{n-1}} \rho_v'(0) d v \\
	&= \rho_0^{n-1} \omega_n^{-1/n-1} I_K \frac {L(s_0)} {L'(s_0)} n \omega_n- \frac 1n \omega_n ^{-1} I_K  \rho_0^n n \omega_n\\
	&- \rho_0^{n-1} \omega_n ^{-1/n} L'(s_0)^{-1} \int_{S^{n-1}}  \frac{\partial}{\partial t} g_{K_t}(s_0 v) \bigg|_{t=0} d v.
\end{align}

	By \eqref{eq_d1_g} in Proposition \ref{res_d1_g} we have $ \frac{\partial}{\partial t} g_{K_t}(s_0 v) \bigg|_{t=0} = W_{K,v}(s_0)$.
	Observe that $S(x)$ is a union of two spherical caps, so for $v,w \in S^{n-1}$, we have $v \in S(s_0 w)$ if and only if $w \in S(s_0 v)$, then
\begin{align}
	\label{eq_averaged_S}
	\int_{S^{n-1}} W_{K,v}(s_0) d v 
    &= \int_{S^{n-1}} \int_{S^{n-1}} \chi_{S(s_0 v)}(w) \rho_K(w) d w d v \\
    &= \int_{S^{n-1}} \int_{S^{n-1}} \chi_{S(s_0 w)}(v) d v \rho_K(w) d w  \\
	&= S(s_0) I_K.
\end{align}
We get
\begin{align}
     \frac{\partial}{\partial t} \vol{C_\delta \bKt} \bigg|_{t=0}
	&= \rho_0^{n-1} \omega_n ^{-1/n} n I_K \frac {L(s_0)} {L'(s_0)} -  I_K  \rho_0^n
	- \rho_0^{n-1} \omega_n ^{-1/n} L'(s_0)^{-1} S(s) I_K  \\
&= I_K L'(s_0)^{-1} \omega_n^{-1/n} \rho_0^{n-1} 
	(n L(s_0) - s_0 L'(s_0) - S(s_0) ). \\
\end{align}
Finally we shall prove that
\[n L(s_0) = s_0 L'(s_0) + S(s_0)\]
which concludes the proof.

Consider the $n-1$ dimensional circle $S_2 = (\frac 12 s_0 v + v^\perp) \cap (\B+s_0 v)$ and observe that $L'(s_0) = \vol[n-1]{S_2}$.
Consider the cone $D_2$ with vertex at the origin and base $S_2$.
Using the cone volume measure (see (9.33) of \cite{Sh1}), this is, $\frac 1n |\langle n(x), x \rangle| d S(x)$ (where $n$ is a unit normal vector to the surface) to compute the volumes of the cones we get
\begin{align}
	L(s_0)
	&= 2(\vol{D(x)} - \vol{D_2}) \\
	&= \frac 2n \left( \int_{C(x)} 1 d S(x)  - \int_{S_2} \frac{s_0}2 d S(x)\right) \\
	&= \frac 2n ( \frac 12 S(s) - \frac{s_0}2 L'(s_0) ).
\end{align}
and the proof is complete.
\end{proof}

\section{Second-order Taylor Expansion of $C_\delta \bKt$ in the plane}
\label{sec_second_order}

In order to compute the second derivative of $\vol{C_\delta \bKt}$ we need a second-order estimate of the covariogram of $K_t$.
From now on, all computations will be made for $n=2$.
We will make use of Proposition \ref{res_serious_approximation} again. In dimension $2$, the set $T(c t)$ is a union of two closed balls.
\begin{proposition}
	\label{res_d2_g}
	Let $K \subseteq \R^2$ be a planar radial set and $K_t$ be the radial body defined by \eqref{eq_def_Kt}, then
	\begin{equation}
	\label{eq_d2_g}
		g_{K_t}(x) = L(x) + t \int_{S(x)} \rho_K(v) d v + t^2 \frac 12 \int_{S(x)} \rho_K(v)^2 d v + t^2 T_K(x) + o(t^2) \\ 
	\end{equation}
	where $\frac{o(t^2)}{t^2} \to 0$ as $t \to 0^+$, for fixed $K$ and $x$, and
	\begin{multline}
		T_K(x) 
		= \frac 1{2 |x| \sqrt{4 - |x|^2}} \Big( 4(\rho_K(v_1)\rho_K(v_2) + \rho_K(v_3)\rho_K(v_4)) \\
		+(\rho_K(v_1)^2+\rho_K(v_2)^2+\rho_K(v_3)^2+\rho_K(v_4)^2)(|x|^2-2)\Big)
	\end{multline}
	where $(v_1, v_2)$ are the boundary points of $S(x)$ and $(v_3, v_4)$ are the lower ones, as shown in Figure \ref{fig_4points}.

	Moreover,
	\begin{equation}
	\label{eq_d2_k}
	\vol{K_t} = \pi + t I_K + t^2 \vol{K}.
	\end{equation}
\end{proposition}

\begin{proof}
	Without loss of generality we may assume $x = (|x|, 0)$.
	Let $p_+, p_-$ be the upper and lower intersection points of $S^1$ and $S^1 + x$.
	Proposition \ref{res_serious_approximation} provides constants $t_0 ,c > 0$ sufficiently small such that outside the balls $B(p_\pm, c t)$, the sets in the left and right of \eqref{eq_separation} are equal for all $t \in [0,t_0]$. This is,
\begin{multline}
	\label{eq_separation_2}
	G_{K_t}(x)  \setminus B(p_+, c t) \setminus B(p_-, c t) \\
	= \left( L(x) \sqcup [1, \rho_{K_t}]_{C(x)} \sqcup [1,\rho_{K_t}]_{C(-x)}^x \right)
	\setminus B(p_+, c t) \setminus B(p_-, c t).
\end{multline}
(see Figure \ref{fig_torus})
	\begin{figure}
		\centering
		\begin{subfigure}[t]{0.45\textwidth}
			\caption{Sets of equation \eqref{eq_separation_2} and $B(p_+, c t)$.}
			\label{fig_torus}
			\includegraphics[width=\textwidth]{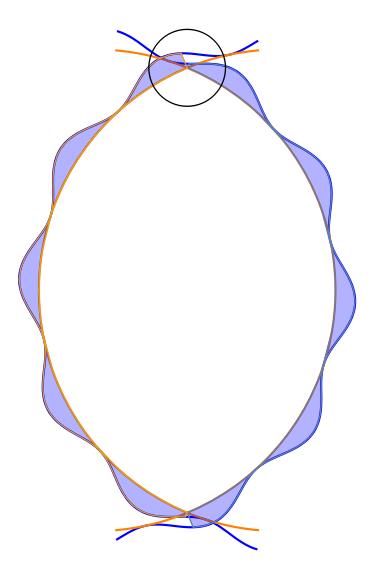}
		\end{subfigure}
		\hfill
		\begin{subfigure}[t]{0.45\textwidth}
			\caption{Boundary points of $S(x)$,\\$v_1, v_2, v_3, v_4$}
			\label{fig_4points}
			\includegraphics[width=\textwidth]{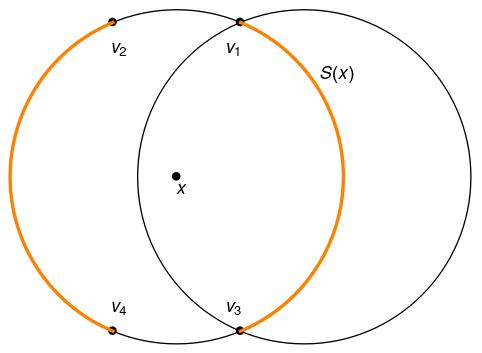}
		\end{subfigure}
	\end{figure}

	To simplify the computations, we will only compute the volume of $G_{K_t}(x)$ intersected with the upper half-plane $H^+$.
	For any measurable $A \subseteq \R^2$, we denote $\vol[2+]{A} = \vol[2]{A \cap H^+}$.
	The intersection with the lower half-plane is similar and will be omitted.
	For small $t>0$, $B(p_+, c t)$ lies in $H^+$.

	To compute the second order term inside the ball we use a blow-up argument at the point $p_+$. The set
	\[R_1(t) = \frac 1t ((G_{K_t}(x) \setminus L(x) \cap B(p_+, c t))-p_+) \]
	is uniformly bounded with respect to $t$, and converges in the Hausdorff metric to
	\begin{multline}
		\label{eq_R1}
		R_1(0) = \Big\{y \in \R^2 : |y|\leq d, \max\{ y . v_\alpha, y . v_{\pi - \alpha} \} \geq 0, \\ y. v_\alpha \leq \rho_K(\alpha), y.v_{\pi-\alpha} \leq \rho_K(\pi - \alpha) \Big\}
	\end{multline}
	(see Figure \ref{fig_intersection3}).

	On the other hand, the set
	\[R_2(t) = \frac 1t \Big( ( [1,\rho_{K_t}]_{C(x)} \sqcup [1,\rho_{K_t}]_{C(-x)}^x - p_+ ) \cap B(p_+, c t)\Big)\]
	is also uniformly bounded with respect to $t$ and converges in the Hausdorff metric to
	\begin{multline}
		\label{eq_R2}
		R_2(0) 
	= \bigg\{y \in \R^2 : |y|\leq d, (0 \leq y.v_\alpha \leq \rho_K(v_\alpha), y.v_{\alpha-\pi/2} \geq 0 ) \\
	\text{ or } (0 \leq y.v_{\pi-\alpha} \leq \rho_K(v_{\pi-\alpha}), y.v_{\pi-\alpha+\pi/2} \geq 0) \bigg\}
	\end{multline}
	(see Figure \ref{fig_intersection1}).

	We get from \eqref{eq_separation_2}, \eqref{eq_R1}, \eqref{eq_R2} and \eqref{eq_d2_C} that
	\begin{align}
		\Big| G_{K_t}(x)  \Big|_{2+}
		&= \vol[2+]{ G_{K_t}(x) \setminus L(x)  \setminus B(p_+, c t) } + \vol[2]{t R_1(t)} + \vol[2+]{L(x)} \\
		&= \vol[2+]{ \left( [1,\rho_{K_t}]_{C(x)} \sqcup [1,\rho_{K_t}]_{C(-x)}^x \right) \setminus B(p_+, c t)  } + \vol[2]{t R_2(t)} + |L(x)|\\
		&+ t^2(\vol[2]{R_1(t)} - \vol[2]{R_2(t)}) \\
		&= \vol[2+]{  [1,\rho_{K_t}]_{C(x)} \sqcup [1,\rho_{K_t}]_{C(-x)}^x  } + \vol[2+]{L(x)}\\
		&+ t^2(R_1(0) - R_2(0)) + o(t^2) \\
		&= \vol[2+]{L(x) } + t \int_{S(x) \cap H_+} \rho_{K_t}(v) d v + \frac 12 t^2 \int_{S(x) \cap H_+} \rho_{K_t}(v)^2 d v \\
		&+ t^2(\vol[2]{R_1(0)} - \vol[2]{R_2(0)}) + o(t^2)
	\end{align}

	\begin{figure}
		\centering
		\begin{subfigure}[b]{0.32\textwidth}
			\caption{Approximation of the set $R_1$ near $p_+$.}
			\label{fig_intersection1}
			\includegraphics[width=\textwidth]{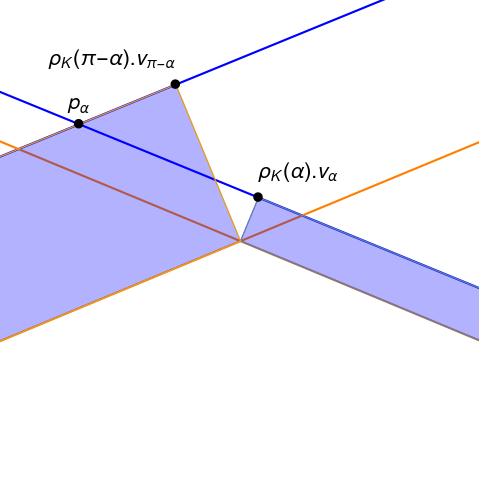}
		\end{subfigure}
		\hfill
		\begin{subfigure}[b]{0.32\textwidth}
			\caption{Difference between $R_1$ and $R_2$.}
			\label{fig_intersection2}
			\includegraphics[width=\textwidth]{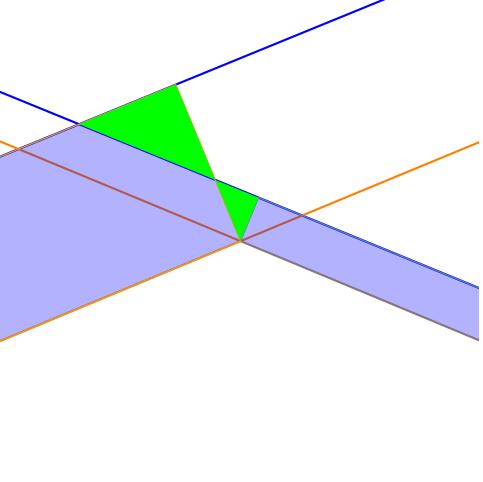}
		\end{subfigure}
		\hfill
		\begin{subfigure}[b]{0.32\textwidth}
			\caption{Approximation of the set $R_2$ near $p_+$.}
			\label{fig_intersection3}
			\includegraphics[width=\textwidth]{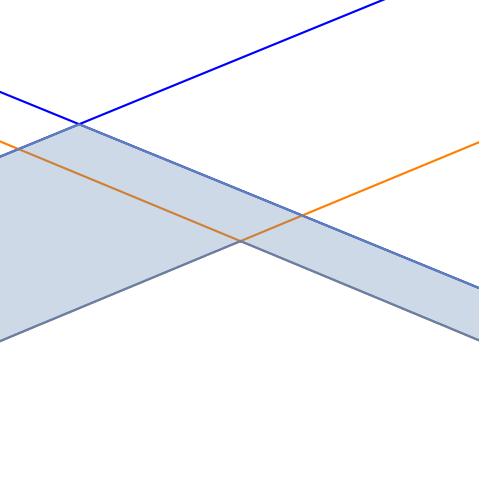}
		\end{subfigure}
	\end{figure}

	The difference $\vol[2]{R_1(0)} - \vol[2]{R_2(0)}$ is exactly the signed area of the quadrilateral with (ordered) vertices $0,\rho_{K_t}(\alpha) v_\alpha, p_\alpha, \rho_{K_t}(\pi-\alpha) v_{\pi-\alpha}$, where $p_\alpha$ is the intersection point of the two lines $p_\alpha \cdot v_\alpha = \rho_{K_t}(\alpha)$ and $p_\alpha \cdot v_{\pi-\alpha} = \rho_{K_t}(\pi-\alpha)$. The sign of the area of each region bounded by the quadrilateral is given by the sens of rotation of the boundary around it. (see Figure \ref{fig_intersection2})
	This quadrilateral is always convex if $\alpha \in (0,\pi/4)$, but for $\alpha > \pi/4$ it becomes self-intersecting as depicted in Figure \ref{fig_intersection2}.
	It is clear that this signed area is exactly $\frac 12( \det(\rho_{K_t}(\alpha) v_\alpha, p_\alpha) + \det(p_\alpha, \rho_{K_t}(\pi-\alpha) v_{\pi-\alpha})$.
	By computing $p_\alpha$ in terms of $\alpha, \rho_{K_t}(\alpha)$ and $\rho_{K_t}(\pi-\alpha)$, and adding the corresponding term for the lower half space, we obtain the formula for $T_K(x)$.

	The second formula is computed easily as
\begin{align}
	\vol{K_t}
	&= \frac 12 \int_{S^1} \rho_{K_t}(v)^2 d v\\
	&= \frac 12 \int_{S^1} (1 + 2 t \rho_{K}(v) + t^2 \rho_K(v)^2 ) d v\\
	&= \pi + t \int_{S^1} \rho_{K}(v) d v + t^2 \vol{K} .\\
\end{align}

\end{proof}

We are ready to compute the second derivative of the volume, in dimension $2$.

\begin{proposition}
	\label{res_d2_vol}
	Let $K$ be a $C^2$ smooth radial set and $\bKt$ be defined by \eqref{eq_def_Ktbar}, then if $\alpha$ is given by \eqref{eq_relations_alpha_delta_s},
\begin{align}
	\label{eq_d2_vol}
	 \frac{\partial^2}{\partial t^2} \vol[2]{C_\delta \bKt} \bigg|_{t=0}
	&= \frac 1{\pi \sin(\alpha )^2} \Biggl(- \frac 1{2\pi} \left(\frac{\sin (2 \alpha )-2 \alpha}{ \sin(\alpha )}\right)^2 I_K^2 \\
	&-\frac 12 \frac{\cos(\alpha)}{\sin(\alpha)} \int_{S^1}\left[ \rho_K(v + \frac\pi2 \pm \frac\pi2\pm \alpha) \right] w_{K,v}(\alpha) d v \\
	&+ \frac 1{4\sin(\alpha)^2} \int_{S^1}\left(w_{K,v}(\alpha) \right)^2 d v + 2 \int_{S^1}\rho_K(v-\alpha +\pi ) \rho_K(v+\alpha )d v\\
	&+ 4 \cos (2 \alpha ) \vol[2]{K} \Biggr)
\end{align}
	where 
	\[\left[ \rho_K(v + \frac\pi2 \pm \frac\pi2\pm \alpha) \right] = \rho_K(v + \alpha) + \rho_K(v - \alpha) + \rho_K(v + \pi + \alpha) + \rho_K(v + \pi - \alpha).\]
\end{proposition}
\begin{proof}

	First we compute the second derivative of $\vol{C_\delta \bKt}$ with respect to $t$ using \eqref{eq_radial_volume}, at $t=0$.
	\begin{equation}
		\label{eq_d2_vol_naked}
		 \frac{\partial^2}{\partial t^2} \vol[2]{C_\delta \bKt} \bigg|_{t=0} = \int_{S^1} \rho_v'(0)^2 d v + \rho_0 \int_{S^1} \rho_K''(0) d v.
	\end{equation}

	In order to simplify the computations we write \eqref{eq_rho} as 
	\begin{equation}
		\label{eq_rho_s}
		\delta = k(t)^{-1} g_v(t, s_v(t))
	\end{equation}
	where $s_v(t) = k(t)^{1/2} \rho_v(t)$, and take derivative with respect to $t$ at $t=0$, to obtain
	\begin{equation}
		\label{eq_d1_s}
		s_v'(0)=\frac {k_0' g_{v,0}-k_0 \partial_t g_{v,0}}{k_0 \partial_s g_{v,0}}
	\end{equation}
	where $g_{v,0} = g_v(0, s_v(0))$, $\partial_t g_{v,0} = \partial_t g_v(0, s_v(0))$.
	Take the second derivative of \eqref{eq_rho_s} with respect to $t$, at $t=0$, and use \eqref{eq_d1_s} to get
	\begin{align}
		\label{eq_d2_s}
		s_v''(0)
		&=\frac 1{k_0^2 \left(\partial_s g_{v,0}\right)^3} \Biggl( k_0 g_{v,0} \left(-2 k_0' \partial_s g_{v,0} \partial_{t,s}g_{v,0}+2 k_0' \partial_{s,s}g_{v,0} \partial_t g_{v,0}+k_0'' \left(\partial_s g_{v,0}\right)^2\right) \\
		&-\left(\partial_t g_{v,0}\right)^2 k_0^2 \partial_{s,s}g_{v,0} +2 k_0^2 \partial_s g_{v,0} \partial_t g_{v,0} \partial_{t,s}g_{v,0}-k_0^2 \left(\partial_s g_{v,0}\right)^2 \partial_{t,t}g_{v,0} \\
		&-\left(k_0'\right)^2 g_{v,0}^2 \partial_{s,s}g_{v,0}\Biggr)
	\end{align}
	The terms $\rho_v', \rho_v''$ can be computed from $s_v', s_v''$ by the relation $s_v(t) = k(t)^{1/2} \rho_v(t)$ as
	\begin{equation}
		\label{eq_d1_rho_s}
		\rho_v'(0) = s_v'(0)k_0^{-1/2} - \frac 12 s_v(0) k_0^{-3/2} k_0' \\
	\end{equation}
	and
	\begin{equation}
		\label{eq_d2_rho_s}
		\rho_v''(0) = -\frac {k_0' s_v'(0)}{k_0^{3/2}}+\frac {s_v''(0)}{\sqrt{k_0}}+\frac {3 \left(k_0'\right)^2 \rho_v(0)}{4 k_0^2}-\frac {k_0'' \rho_v(0)}{2 k_0}.
	\end{equation}

	Using Propositions \ref{res_d1_g} and \ref{res_d2_g} we get the following identities:
	\begin{equation}
		\label{eq_identities_1}
		\begin{array}{ccc}
		k_0 = \pi & k_0' = I_K & k_0'' = 2 \vol[2]{K} \\
		g_{v,0} = L(s_0) & \partial_s g_{v,0} = L'(s_0) & \partial_s g_{v,0} = L''(s_0) \\
	\end{array}
	\end{equation}
	\begin{equation}
		\label{eq_identities_2}
		\begin{array}{cc}
			\partial_t g_{v,0} = W_{K,v}(s_0)  & \partial_{t,t} g_{v,0} = \int_{S(s_0 v)} \rho_K(w)^2 d w + 2 T_K(s_0 v) \\
			\partial_{t,s} g_{v,0} = W_{K,v}'(s_0).
		\end{array}
	\end{equation}
	Notice that $g_{v,0}, \partial_s g_{v,0}, \partial_{s,s} g_{v,0}$ are independent of $v$.

	Integrating \eqref{eq_d2_s},
	\begin{align}
		\label{eq_int_d2_s}
		\int_{S^1} s_v''(0) d v
		&= \frac 1{\pi^2 (L')^3} \bigg[ -2\pi I_K L L' \int_{S^1} W_{K,v}' d v + 2\pi L I_K L'' \int_{S^1} W_{K,v} d v  \\
		& + 4\pi^2 L \vol[2]{K} (L')^2 -\pi^2 L'' \int_{S^1} W_{K,v}^2 d v  - 2\pi I_K^2 L^2 L'' \\
		& + \pi^2 (L')^2 \int_{S^1} \Big( \int_{S(s_0 v)} \rho_K(w)^2 d w + 2T_K(s_0 v)\Big)  d v\\
		& + 2\pi^2 L' \int_{S^1}W_{K,v} W_{K,v}' d v \bigg]
	\end{align}
	where we omitted the argument $s_0$ in the functions $W_{K,v}, W_{K,v}', L, L'$ and $L''$.

	Using a computation similar to \eqref{eq_averaged_S}, we obtain 
	\begin{equation}
		\label{eq_int_w2}
		\int_{S^1}\int_{S(s_0 v)} \rho_K(w)^2 dw dv = 2 S(s_0) \vol[2]{K}.
	\end{equation}
	Also from \eqref{eq_averaged_S},
	\begin{equation}
		\label{eq_int_wprime}
		\int_{S^1} W_{K,v}'(s_0) d v = S'(s_0) I_K.
	\end{equation}

	We combine \eqref{eq_d2_vol_naked}, \eqref{eq_d1_s}, \eqref{eq_d2_s}, \eqref{eq_d1_rho_s}, \eqref{eq_d2_rho_s} the identities \eqref{eq_identities_1} and \eqref{eq_identities_2}, and \eqref{eq_int_w2}, \eqref{eq_int_wprime}, and we get
	\begin{align}
	\label{eq_d2_vol_s_with_d2s}
 \frac{\partial^2}{\partial t^2} \vol[2]{C_\delta \bKt} \bigg|_{t=0}
		&= \frac 2{\pi^2} \left(\frac {I_K L}{\pi^{3/2} L'} \right)^2 + \frac 1{\pi (L')^2} \int_{S^1} W_{K,v}^2 d v + \frac 1{2\pi}(\rho_0 I_K)^2 \\
		&-2\frac{I_K^2 L S}{\pi^2 (L')^2}- \frac 2{\pi^{3/2}} \frac{\rho_0 I_K^2 L}{L'} + \frac{\rho_0}{\pi^{3/2}} \frac{I_K^2 S}{L'} + \rho_0 \bigg[-\frac 1{\pi^{3/2}} \frac{I_K^2 L}{L'}  \\
		& + \frac{I_K^2 S}{\pi^{3/2} L'} + \frac 1{\pi^{1/2}} \int_{S^1} s_v''(0) dv +\frac 3{2\pi} \rho_0 I_K^2 - 2 \rho_0 \vol[2]{K} \bigg].
	\end{align}

	We combine \eqref{eq_d2_vol_s_with_d2s}, with \eqref{eq_int_d2_s}, \eqref{eq_d2_s}, the identities \eqref{eq_identities_1} and \eqref{eq_identities_2}, and after lengthy but straight-forward computations we get
\begin{align}
	\label{eq_d2_vol_s}
	\frac{\partial^2}{\partial t^2} \vol[2]{C_\delta \bKt}& \Big|_{t=0}
	= \frac 1{\pi ^2 (L')^3} \bigg[2 \pi  s_0 L' \left(\int_{S^1}W_{K,v} W_{K,v}'d v\right)  \\
	&+\pi  \left(L'-s_0 L''\right) \left(\int_{S^1}W_{K,v}^2d v\right) -2 \pi  s_0 (L')^2 \left(\int_{S^1}T(v,s_0)d v\right) \\
	&-2 \pi  s_0^2 | K|_2 (L')^3+4 \pi  s_0 \vol[2]K L (L')^2 +2 s_0^2 I_K^2 (L')^3 \\
	& -2 s_0 I_K^2 L S' L'+2 s_0 I_K^2 S L L''+2 s_0 I_K^2 S (L')^2-2 I_K^2 S L L' \\
	&-2 s_0 I_K^2 L^2 L''-4 s_0 I_K^2 L (L')^2+2 I_K^2 L^2 L'-2 \pi  s_0 S (L')^2 \vol[2]{K} \bigg].
   \end{align}

Now we parametrize with respect to the variable $\alpha \in (0,\pi/2)$ with $s_0 = 2\cos(\alpha)$.
We have the following relations:
\begin{equation}
	\label{eq_relations_alpha}
	\begin{array}{cccc}
	s_0 = 2 \cos(\alpha) & 	\delta = L(s_0)/\pi \\
	S(s_0) = 4\alpha & S'(s_0) = -\frac 2{\sin(\alpha)} \\
	L'(s_0) = -2 \sin(\alpha) & 	L''(s_0) = \tan(\alpha)^{-1} & L(s_0) = 2 (\alpha - \cos(\alpha)\sin(\alpha))
	\end{array}
\end{equation}

To compute the term $W_{K,v}'$, observe that 
\begin{align}
	\label{eq_d1_W}
	-2 \sin(\alpha) & W_{K,v}'(2 \cos(\alpha)) 
	= w_{K,v}'(\alpha) \\
	&= \frac {\partial}{\partial \alpha} \left( \int_{v-\alpha}^{v+\alpha} \rho_K(w) d w + \int_{v+\pi-\alpha}^{v+\pi+\alpha} \rho_K(w) d w \right) \\
	&= \rho_K(v + \alpha) + \rho_K(v - \alpha) + \rho_K(v + \pi + \alpha) + \rho_K(v + \pi - \alpha) \\
	&= \left[ \rho_K(v + \frac\pi2 \pm \frac\pi2\pm \alpha) \right].
\end{align}
Using the identities \eqref{eq_relations_alpha} and \eqref{eq_d1_W} we simplify equation \eqref{eq_d2_vol_s} to obtain \eqref{eq_d2_vol}.

\end{proof}

We are ready to compute the counterexample:

\begin{proof}[Proof of Theorem \ref{res_cos_m_perturbation}]

Consider the (infinitely smooth and symmetric) radial set given by 
\begin{equation}
	\label{eq_def_Km}
	\rho_{K^m}(v) = \cos(m v)^2 = \frac 12 \cos(2 m v)+\frac 12.
\end{equation}
	(see Figure \ref{fig_setsKm}.)
	\begin{figure}
		\centering
		\caption{The sets $K_t^m$ for $m=1, 2, 3, 4$ and $t = \frac 1{2 m^2}$.}
		\label{fig_setsKm}
		\includegraphics[width=.9\textwidth]{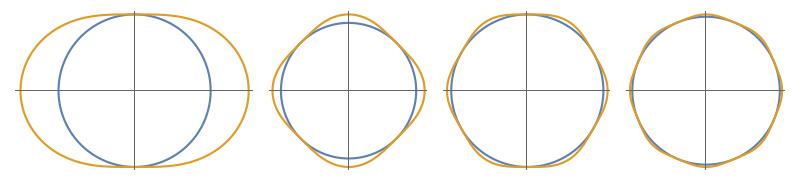}
	\end{figure}

All integrals in \eqref{eq_d2_vol} can be computed exactly using the two expressions for $\rho_{K^m}(v)$ in \eqref{eq_def_Km}.
To compute the last integral we use the identity
\[\cos(m(v-\alpha))^2 \cos(m(v+\alpha))^2 = \frac 14(\cos(2 m v) + \cos(2 m \alpha))^2.\]
Integrating every term in \eqref{eq_d2_vol} we obtain

\begin{equation}
	I_K = \pi, \quad \vol[2]{K} = \frac 38 \pi, \quad \int_{S^1}\rho_K(v-\alpha +\pi ) \rho_K(v+\alpha )d v = \frac 14 \pi  (\cos (4 \alpha  m)+2) \\
\end{equation}
\begin{align}
	w_{K,v}(\alpha) &= \frac 1{4 m} \left(\sin (2 m (\alpha +v))+\sin (2 m (\alpha +v+\pi )) - 2 \sin (2 m (v-\alpha ))\right) \\
	\int_{S^1} w_{K,v}(\alpha)^2 d v &= \frac{\pi  \left(16 \alpha ^2 m^2-\cos (4 \alpha  m)+1\right)}{2 m^2} \\
\end{align}
and
\[
	\int_{S^1}\left[ \rho_K(v + \frac\pi2 \pm \frac\pi2\pm \alpha) \right] w_{K,v}(\alpha) d v = \frac 1m \pi  (8 \alpha  m+\sin (4 \alpha  m)).
\]

Denote $F_m(\alpha) =  \frac{\partial^2 }{\partial t^2} \vol[2]{C_\delta \overline {K^m_t}} \bigg|_{t=0}$, where $\delta$ and $\alpha$ are related by \eqref{eq_relations_alpha_delta_s}.
Putting all the integrals together we get
\begin{align}
	\label{eq_def_Fm}
	F_m(\alpha) &= \frac 1{\sin(\alpha)^2}\bigg( \frac 12 \cos(2\alpha) + \frac 12 \cos (4 \alpha  m) \\ 
	& +\frac 1{8 m^2 \sin ^2(\alpha )}-\frac{\cos (4 \alpha  m)}{8 m^2 \sin ^2(\alpha)}-\frac{\sin (4 \alpha  m)}{2 m \tan (\alpha )} \bigg).
\end{align}
Every pair $m, \alpha$ for which $F_m(\alpha)$ is positive will provide us a counterexample to Question \ref{thequestion}.
To finish the proof, it remains to prove that for every $\alpha_0 \in (0,\pi/2)$ there exists $m$ such that $F_m(\alpha_0) > 0$.
    
Consider 
\[c(\alpha, m) = \frac 1{8 m^2 \sin(\alpha)^2}-\frac{\cos (4 \alpha  m)}{8 m^2 \sin ^2(\alpha)}-\frac{\sin (4 \alpha  m)}{2 m \tan (\alpha )}.\]
Equation \eqref{eq_def_Fm} can be written as
\[\sin(\alpha)^2 F_{S^m}(\alpha) = \frac 12 \cos(2\alpha) + \frac 12 \cos (4 \alpha  m) + c(\alpha, m).\]
(see Figure \ref{fig_oscilatingfunction})

	\begin{figure}
		\centering
		\caption{The function $\sin(\alpha)^2 F_m(\alpha)$ for $m=10$.}
		\label{fig_oscilatingfunction}
		\includegraphics[width=.5\textwidth]{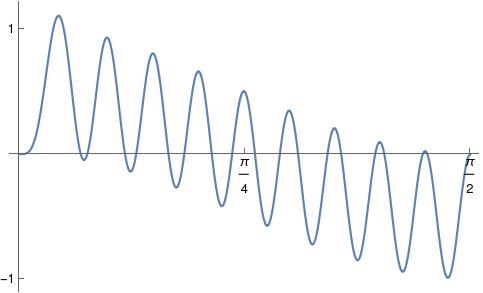}
	\end{figure}

The function $c(\alpha, m)$ tends to $0$ as $m \to \infty$, for every $\alpha \in (0, \pi/2]$.

Fix $\alpha_0 \in (0,\pi/2)$ and consider $m_0$ such that for every $m \geq m_0$, $\frac 12 \cos(2\alpha_0) + c(\alpha_0, m) > -1/2$.
This is possible since $\cos(2\alpha_0) \in (-1,1)$.

If $\alpha_0/\pi$ is a rational number, choose a suitable $m \geq m_0$ such that $\alpha_0 m/\pi$ is an integer, then $\cos(4 m \alpha_0) = 1$.
If $\alpha_0/\pi$ is not rational, the sequence $\cos(4 m \alpha_0), m\geq m_0$ is dense in $[-1,1]$ and we may choose $m \geq m_0$ so that $\cos(4 m \alpha_0)$ is arbitrarily close to $1$.

In both cases we obtain at least one value of $m$ such that $F_m(\alpha_0) > 0$.
\end{proof}

Finally we are ready to give a negative answer to Question \ref{thequestion}.
\begin{proof}[Proof of Theorem \ref{res_negative_answer}]
	Let $\delta \in (0,1)$ and take the value of $\alpha$ given by the relations \eqref{eq_relations_alpha}.
	By the proof of Theorem \ref{res_cos_m_perturbation} there is $m \in \mathbb N$ such that $F_m(\alpha) > 0$.
	Consider the radial set $K^m$ defined by \eqref{eq_def_Km} and the radial body $K^m_t = (K^m)_t$ defined by \eqref{eq_def_Kt}.
	Since the function $\rho_{K^m_t}$ converges to $1$ in the $C^\beta$ topology for every $\beta \geq 0$, and since convexity is a $C^2$ property of $\rho$, there exists $\tilde t_m>0$ such that $K_t^m$ is convex for every $t \in [0, \tilde t_m]$ (by analyzing the Gauss curvature, one can see that $K_t^m$ is convex for $t \in [0,\frac 1{2m^2}]$).
	By Theorems \ref{res_first_variation} and \ref{res_cos_m_perturbation}, there exists $t_m < \tilde t_m$ such that the function $t \mapsto \vol[2]{C_{\delta} \overline{K^m_t}}$ is increasing at $[0, t_m]$. Then
	\[\vol[2]{C_{\delta} \overline{K^m_{t_m}}} > \vol[2]{C_{\delta} \overline{K^m_0}} = \vol[2]{C_{\delta} \overline{\B}},\]
	and the proof is complete.
\end{proof}

\section{The limit as $\delta \to 1$}
\label{sec_limit}

In this section we prove Theorem \ref{res_positive_almost}.
We split the proof in two parts: first we compute the limit as $\delta \to 1$, and later we show that the limit is non-positive.

Theorem \ref{res_positive_almost} is a direct consequence of Propositions \ref{res_limit} and \ref{res_limit_negative}.

\begin{proposition}
	\label{res_limit}
	Let $K$ be a $C^2$ smooth radial set, then
	\begin{align}
		\lim_{\delta \to 1^-} \frac 1{(1-\delta)^2}  \frac{\partial^2}{\partial t^2} \vol[2]{C_\delta \bKt} \bigg|_{t=0}
		&= \frac 34 \pi \int_{S^1} \left(\frac{\rho_K(v) + \rho_K(v+\pi)}2\right)^2 d v \\
		&-\frac 12 \left( \int_{S^1} \rho_K(v) d v\right)^2 -\frac\pi4 \int_{S^1}\rho_K'(v)^2 d v + \frac \pi2 \vol{K}.
		\label{eq_limit}
	\end{align}

\end{proposition}
\begin{proof}
	First notice that 
	\[ \lim_{\alpha \to \pi/2^-} \frac{1-\delta} {\frac 4\pi \cos(\alpha)} = 1,\]
	where $\delta$ and $\alpha$ are related by \eqref{eq_relations_alpha_delta_s},
	so we may replace the factor $(1-\delta)^2$ in \eqref{eq_limit} by $(\frac 4\pi \cos(\alpha))^2$.
	We rearrange some terms of \eqref{eq_d2_vol} to obtain
	\begin{align}
		\label{eq_lim_splitted}
		\frac {\pi \sin(\alpha)}{\cos(\alpha)^2}  &\frac{\partial^2}{\partial t^2} \vol[2]{C_\delta \bKt} \bigg|_{t=0}\\
		&= \frac 1 {\sin(\alpha)} A_1(\alpha) + 2 A_2(\alpha) - \frac 1{\sin(\alpha)^2} A_3(\alpha) + 8 \vol[2]{K} - \frac 2\pi I_K^2
	\end{align}
where
\begin{align}
	A_1(\alpha) &=\frac 1 { \cos(\alpha)}\left[ \frac {4\alpha}\pi I_K^2 - \frac 12 \int_{S^1} \left[ \rho_K(v + \frac\pi2 \pm \frac\pi2\pm \alpha) \right] w_{K,v}(\alpha) d v  \right] \\
	A_2(\alpha) &= \frac 1{\cos(\alpha)^2}\left[ \int_{S^1}\rho_K(v-\alpha +\pi) \rho_K(v+\alpha)d v - \int_{S^1} \rho(v)^2 d v \right]\\ 
	A_3(\alpha) &= \frac 1{\cos(\alpha)^2}\left[ \frac{2 \alpha^2}{\pi} I_K^2 - \frac 14 \int_{S^1}w_{K,v}(\alpha)^2 d v \right]. 
\end{align}
	Here we used the identities $\cos(2\alpha) = 2\cos(\alpha)^2 - 1$ and $\sin(2\alpha) = 2 \sin(\alpha) \cos(\alpha)$.
To compute the limits, first we observe that
	\[\lim_{\alpha\to\pi/2^-} \frac {I_K - w_{K,v}(\alpha)}{2\pi - 4\alpha} = \frac 12(\rho(v+\pi/2) + \rho(v-\pi/2)),\]
	since the left term is the average of $\rho_K$ in the complement of $S(\cos(\alpha) v)$.

	We compute $A_1$:
	\begin{align}
		\cos(\alpha) &A_1(\alpha)
		= \frac{4 \alpha}\pi I_K^2 + \frac 12 (2\pi - 4\alpha) \int_{S^1} \left[ \rho_K(v + \frac\pi2 \pm \frac\pi2\pm \alpha) \right] \frac{I_K - w_{K,v}(\alpha)}{2\pi - 4\alpha} d v \\
		&- \frac 12 \int_{S^1} \left[ \rho_K(v + \frac\pi2 \pm \frac\pi2\pm \alpha) \right] d v I_K \\
		&=\frac 4\pi I_K^2(\alpha - \pi/2) + 2(\pi/2 - \alpha) \int_{S^1} \left[ \rho_K(v + \frac\pi2 \pm \frac\pi2\pm \alpha) \right] \frac{I_K - w_{K,v}(\alpha)}{2\pi - 4\alpha} d v.
	\end{align}
	Since $\frac{\cos(\alpha)}{\pi/2 - \alpha} \to 1$ when $\alpha \to \pi/2$ we obtain
	\[\lim_{\alpha \to \pi/2^-} A_1(\alpha) = 8 \int_{S^1} \left(\frac{\rho_K(v) + \rho_K(v+\pi)}2\right)^2 d v - \frac 4\pi I_K^2.\]

	We compute $A_2$:

	\begin{align}
		\cos(\alpha)^2 &A_2(\alpha)
		= \int_{S^1}\rho_K(v-\alpha +\pi) \rho_K(v+\alpha)d v - \int_{S^1}\rho_K(v)^2 d v \\
		&= \frac 12 \int_{S^1} ( 2 \rho_K(v-\alpha +\pi) \rho_K(v+\alpha)d v - \rho_K(v+\alpha)^2 - \rho_K(v-\alpha +\pi)^2 ) d v \\
		&= - 2 (\pi/2-\alpha)^2 \int_{S^1} \left( \frac{\rho_K(v-\alpha +\pi) - \rho_K(v+\alpha)}{\pi-2\alpha} \right)^2 d v \\
	\end{align}
	we get
	\[\lim_{\alpha \to \pi/2^-} A_2(\alpha) = -2 \int_{S^1} \rho_K'(v)^2 d v.\]

	We compute $A_3$:
	\begin{align}
		\cos(\alpha)^2 &A_3(\alpha)
		= \frac{2 \alpha^2}{\pi} I_K^2 - \frac 14 \int_{S^1} (w_{K,v}(\alpha) - I_K)^2 d v - \frac 14 \int_{S^1} ( 2I_K w_{K,v}(\alpha) - I_K^2 ) d v \\
		&= I_K^2 \left( \frac{2 \alpha^2}{\pi} + \frac \pi 2\right) - \frac 12 I_K \int_{S^1} w_{K,v}(\alpha) d v \\
		&- 4 \left( \frac \pi 2 - \alpha \right)^2 \int_{S^1} \left(\frac{w_{K,v}(\alpha) - I_K}{2\pi - \alpha} \right)^2 d v \\
		&= I_K^2 \left( \frac{2 \alpha^2}{\pi} + \frac \pi 2 -2\alpha \right) - 4 \left( \frac \pi 2 - \alpha \right)^2 \int_{S^1} \left(\frac{w_{K,v}(\alpha) - I_K}{2\pi - \alpha} \right)^2 d v \\
		&= \frac 2\pi I_K^2 \left( \frac \pi 2 - \alpha \right)^2 - 4 \left( \frac \pi 2 - \alpha \right)^2 \int_{S^1} \left(\frac{w_{K,v}(\alpha) - I_K}{2\pi - \alpha} \right)^2 d v,
	\end{align}
	and we get
	\[\lim_{\alpha \to \pi/2^-} A_3(\alpha) = \frac 2\pi I_K^2 - 4 \int_{S^1}\left(\frac{\rho_K(v) + \rho_K(v+\pi)}2\right)^2 d v.\]
	
	Putting together all the terms $A_i$ in \eqref{eq_lim_splitted}, we get 
	\begin{multline}
		\lim_{\alpha \to \pi/2^-} \frac{\pi }{4\cos(\alpha)^2} \frac{\partial^2}{\partial t^2} \vol[2]{C_\delta \bKt} \bigg|_{t=0}
		=3 \int_{S^1} \left(\frac{\rho_K(v) + \rho_K(v+\pi)}2\right)^2 d v \\
		-\frac 2\pi \left( \int_{S^1} \rho_K(v) d v\right)^2 -\int_{S^1}\rho_K'(v)^2 d v +2 \vol{K},
	\end{multline}
	and the Proposition is proved.
\end{proof}

Finally, we shall prove that the limit of the second derivative is non-positive.
\begin{proposition}
	\label{res_limit_negative}
	For every $C^2$ smooth radial set $K$,
	\begin{multline}
		3 \int_{S^1} \left(\frac{\rho_K(v) + \rho_K(v+\pi)}2\right)^2 d v -\frac 2\pi \left( \int_{S^1} \rho_K(v) d v\right)^2 \\ -\int_{S^1}\rho_K'(v)^2 d v +2 \vol{K} \leq 0.
	\end{multline}
	Equality holds if and only if 
	\[\rho_K(\alpha) = a + b \cos(\alpha) + c \sin(\alpha) + d \cos(2\alpha) + e \sin(2\alpha)\]
	for some constants $a,b,c,d,e$.
\end{proposition}

\begin{proof}
	Since $\rho_K$ is a real periodic and continuous function we can represent it as a Fourier series
    \[\rho_K(\alpha) = \sum_{n \in \mathbb Z} a_n e^{i n \alpha}\]
    with $a_{-n} = \overline{a_n}, a_0 >0$.
    The integrals are expressed as
	\begin{equation}
		\label{eq_Fourier_identities}
		\int_{S^1} \rho_K = 2\pi a_0,\quad \int_{S^1} \rho_K^2 = 4\pi \sum_{n\geq 1} |a_n|^2 + 2\pi a_0^2, \quad \int_{S^1} (\rho_K')^2 = 4\pi \sum_{n\geq 1} n^2 |a_n|^2
	\end{equation}
	The symmetric part is
	\[ \frac{\rho_K(\alpha) + \rho_K(\alpha+\pi)}2 = \sum_{n \in \mathbb Z} \varepsilon_n a_n e^{i n \alpha}\]
	where $\varepsilon_n=1$ if $n$ is even, and $\varepsilon_n = 0$ if $n$ is odd. Using \eqref{eq_Fourier_identities} we have
	\[ \int_{S^1} \left(\frac{\rho_K(v) + \rho_K(v+\pi)}2\right)^2 d v = 4\pi \sum_{n \geq 1 } \varepsilon_n |a_n|^2 + 2\pi a_0^2\]
	and we compute
    \begin{align}
    \lim_{\alpha \to \pi/2^-} \frac{\pi F(\alpha)}{4\cos(\alpha)^2} 
	    &=3 \left(4\pi \sum_{n \geq 1} \varepsilon_n |a_n|^2 + 2\pi a_0^2 \right) -\frac 2\pi (2\pi a_0)^2 \\
	    &-\left(4\pi \sum_{n\geq 1} n^2 |a_n|^2\right) + \left(4\pi \sum_{n\geq 1} |a_n|^2 + 2\pi a_0^2\right) \\
	    &= 12 \pi \sum_{n \geq 1} \varepsilon_n |a_n|^2 - 4\pi \sum_{n\geq 1} n^2 |a_n|^2 +4\pi \sum_{n\geq 1} |a_n|^2\\
	    &= 4\pi \left(\sum_{n\geq 1} (1+3\varepsilon_n) |a_n|^2 - \sum_{n\geq 1} n^2 |a_n|^2 \right)\\
    \end{align}
    Observe that $1+3\varepsilon_n \leq n^2$ for every $n \geq 1$, then the inequality follows.

    Equality holds if and only if $a_n = 0$ for all $|n| \geq 3$, which happens if and only if $K$ has radial function
	\[\rho_K(\alpha) = a + b \cos(\alpha) + c \sin(\alpha) + d \cos(2\alpha) + e \sin(2\alpha).\]
This function is a polynomial of degree $2$ in two variables, evaluated in $v_\alpha$.
\end{proof}

\section*{Acknowledgments}
 J. Haddad was supported by Grant RYC2021-031572-I, funded by the Ministry of Science and Innovation / State Research Agency / 10.13039 / 501100011033 and by the E.U. NextGeneration EU/Recovery, Transformation and Resilience Plan
\bibliographystyle{abbrv}
\bibliography{references}

\end{document}